\documentclass[a4paper, leqno, 11pt, twoside]{amsart}

\title{An arithmetic Hilbert-Samuel theorem for singular hermitian line bundles and cusp forms}
\author{Robert Berman}
\address{Chalmers Techniska H\"ogskola, G\"oteborg, Sweden}
\email{robertb@chalmers.se}
\author{Gerard Freixas i Montplet}
\address{CNRS, Institut de Math\'ematiques de Jussieu, Paris, France}
\email{freixas@math.jussieu.fr}
\date{}

\usepackage{amsmath, mathrsfs, amssymb, amscd, amsthm, amstext}
\usepackage[]{fontenc}
\usepackage{xy}
\usepackage{enumerate}
\xyoption{all}

\numberwithin{equation}{section}

\theoremstyle{plain}
\newtheorem{theorem}{Theorem}[section]
\newtheorem{proposition}[theorem]{Proposition}
\newtheorem{lemma}[theorem]{Lemma}

\theoremstyle{definition}
\newtheorem{definition}[theorem]{Definition}

\newtheorem{notation}[theorem]{Notation}

\theoremstyle{remark}
\newtheorem{remark}[theorem]{Remark}
\DeclareMathOperator{\Spec}{Spec}

\DeclareMathOperator{\c1}{c_{1}}

\DeclareMathOperator{\SL}{SL}

\DeclareMathOperator{\ari}{ar}

\DeclareMathOperator{\adeg}{\widehat{deg}}

\DeclareMathOperator{\PSH}{PSH}
\DeclareMathOperator{\vol}{vol}
\DeclareMathOperator{\MA}{MA}

\DeclareMathOperator{\rk}{rk}

\DeclareMathOperator{\pet}{Pet}

\newcommand{\OO}{\mathcal{O}}

\newcommand{\CC}{\mathbb{C}}
\newcommand{\Int}{\mathbb{Z}}

\newcommand{\PP}{\mathbb{P}}

\newcommand{\RR}{\mathbb{R}}
\newcommand{\QQ}{\mathbb{Q}}

\newcommand{\pd}{\partial}
\newcommand{\cz}{\overline{z}}

\newcommand{\pX}{\mathscr{X}}

\newcommand{\pL}{\mathcal{L}}

\newcommand{\pE}{\mathcal{E}}

\newcommand{\pF}{\mathcal{F}}
\newcommand{\pT}{\mathcal{T}}
\newcommand{\pK}{\mathcal{K}}
\newcommand{\aL}{\mathscr{L}}

\newcommand{\aK}{\mathscr{K}}

\newcommand{\aM}{\mathscr{M}}
\newcommand{\aN}{\mathscr{N}}
\newcommand{\ov}{\overline}

\newcommand{\subs}{\subsection{}}

\begin{document}
\setcounter{tocdepth}{1}
\setcounter{section}{0}

\begin{abstract}
We prove an arithmetic Hilbert-Samuel type theorem for semi-positive singular hermitian line bundles of finite height. In particular, the theorem applies to the log-singular metrics of Burgos-Kramer-K\"uhn. Our theorem is thus suitable for application to some non-compact Shimura varieties with their bundles of cusp forms. As an application, we treat the case of Hilbert modular surfaces, establishing an arithmetic analogue of the classical result expressing the dimensions of spaces of cusp forms in terms of special values of Dedekind zeta functions.
\end{abstract}

\maketitle

\tableofcontents

\section{Introduction}
Arithmetic intersection theory is an extension of algebraic geometry of schemes over rings of integers of numbers fields, that incorporates complex geometric tools on the analytic spaces defined by their complex points. Its foundations go back to the work of Arakelov \cite{Arakelov}, on an extension of intersection theory to this setting. It was vastly generalized in the work of Gillet-Soul\'e \cite{GS:AIT}, \cite{GS:CC1} and lead to the definition of heights of cycles on arithmetic varieties, with respect to hermitian line bundles, by Bost-Gillet-Soul\'e \cite{BGS}. This is a counterpart in this theory of the notion of geometric degree of a line bundle, and measures the arithmetic complexity of the equations defining a projective arithmetic variety. One of the major achievements in arithmetic intersection theory was the proof of an analogue of the Grothendieck-Riemann-Roch theorem \cite{GS:ARR}, relying on deep results of Bismut and coworkers on analytic torsion. Combining their theorem with work of Bismut-Vasserot \cite{Bismut-Vasserot}, Gillet and Soul\'e were able to derive an analogue of the Hilbert-Samuel theorem, relating the covolumes of lattices of global sections of powers of an hermitian ample line bundle to the height of the variety. An alternative approach avoiding the use of analytic torsion was proposed by Abbes-Bouche \cite{Abbes-Bouche}. In view of its diophantine applications, the arithmetic Hilbert-Samuel theorem has been the object of numerous  generalizations, as for instance in Zhang \cite{Zhang}, Moriwaki \cite{Moriwaki} and Yuan \cite{Yuan}.

Arithmetic intersection theory has also allowed to attach numerical invariants --arithmetic intersection numbers-- to several arithmetic cycles or hermitian vector bundles, for instance those arising from automorphic forms on Shimura varieties. These are involved in the formulation of Kudla's program on generating series of arithmetic intersection numbers and central values of incoherent Eisenstein series, and in conjectures of Maillot-R\"ossler providing an interpretation of logarithmic derivatives of $L$-functions at negative integers \cite{Maillot-Rossler}. A technical difficulty in these statements is the fact that, in general, Shimura varieties are not compact, and need to be compactified. This is achieved through the theory of Baily-Borel \cite{Baily-Borel} and the theory of toroidal compactifications \cite{toroidal}. The latter produces pairs $(X,D)$ formed by a projective variety and a normal crossings divisor $D$ in $X$, the boundary of the compactification. Then the analytic component in arithmetic intersection theory requires an extension: the most natural vector bundles on $X$ come equipped with degenerate hermitian structures, whose singularities are localized along $D$. These are the so called good metrics, introduced by Mumford in \cite{Mumford}. To deal with this difficulty, the arithmetic intersection theory was extended in the work of Burgos-Kramer-K\"uhn \cite{BKK2}-\cite{BKK1}. In particular they introduced a variant of the notion of good hermitian metric, namely the notion of log-singular hermitian metric, for which arakelovian heights can be defined. Their formalism was applied with success in the lines of Kudla's program in Burgos-Bruinier-K\"uhn \cite{BBK}.

Contrary to the arithmetic intersection theory of Gillet-Soul\'e, in the generality of Burgos-Kramer-K\"uhn there is no analogue of the Grothendieck-Riemann-Roch theorem. The obstruction is of an analytic nature: the analytic torsion forms are not well defined when the hermitian structures are singular. The case of modular curves or, more generally, non-compact hyperbolic curves defined over number fields, was studied by the second author. He obtained Riemann-Roch and Hilbert-Samuel type theorems for the sheaves of cusp forms of arbitrary weight \cite{GFM1}-\cite{GFM2}. At the base of his methods are Teichm\"uller theory and the geometry of the Deligne-Mumford compacitifcation of the moduli space of curves. These techniques do not generalize to higher dimensions.

In this article we prove a general arithmetic Hilbert-Samuel theorem for line bundles in adjoint form, i.e. powers of a line bundles twisted with the canonical bundle, and endowed with hermitian metrics with suitable singularities. These are the so called semi-positive metrics of finite energy, appearing in the work of Boucksom-Eyssidieux-Guedj-Zeriahi \cite{BEGZ} (see also \textsection \ref{subsec:finite-energy}), and they form the biggest possible class of singular semi-positive metrics for which the height can be defined and is a finite real number (see Remark \ref{remark:height-extension} below). For instance, this class includes the log-singular metrics of Burgos-Kramer-K\"uhn. For log-singular hermitian line bundles we are actually able to provide a general statement in non necessarily adjoint form. Let $\pX\rightarrow\Spec\Int$ be an arithmetic variety, namely an integral flat projective scheme over $\Int$ with smooth generic fiber $\pX_{\QQ}$. Let $\ov{\aL}$ be a line bundle endowed with a semi-positive metric of finite energy. Assume that $\aL_{\QQ}$ is semi-ample and big, and $\aL$ is nef on vertical fibers. Fix an invertible sheaf $\aK$ such that $\aK_{\QQ}=K_{\pX_{\QQ}}$. For every $k\geq 0$ the cohomology group $H^{0}(\pX,\aL^{\otimes k}\otimes\aK)$ can be equipped with a $L^2$ hermitian structure \textsection \ref{subsection:det-coh}. Also, the height of $\pX$ with respect to $\ov{\aL}$, $h_{\ov{\aL}}(\pX)$, is defined \textsection \ref{subsec:heights}. Our first theorem (Theorem \ref{thm:HS}) is stated as follows.
\begin{theorem}\label{thm:main}
 There is an asymptotic expansion
\begin{displaymath}
 	\adeg H^{0}(\pX,\aL^{\otimes k}\otimes\aK)_{L^2}=h_{\ov{\aL}}(\pX)\frac{k^{d}}{d!}+o(k^{d})\text{ as }k\to+\infty.
\end{displaymath}
\end{theorem}
Let now $D\subset\pX_{\QQ}$ be a divisor with normal crossings. We suppose that the hermitian metric on $\ov{\aL}$ is log-singular, with singularities along $D(\CC)$. Furthermore, let $\ov{\aN}$ be any log-singular hermitian line bundle, with singularities along $D(\CC)$ as well. Fix a smooth volume form on $\pX(\CC)$, invariant under the action of complex conjugation. With respect to this volume form and the metric on $\ov{\aL}^{\otimes k}\otimes\ov{\aN}$, there is a $L^{2}$ hermitian structure on $H^{0}(\pX,\aL^{\otimes k}\otimes\aN)$. In this situation we have our second theorem (Theorem \ref{thm:HS_cor}).

\begin{theorem}\label{thm:main_cor}
There is an asymptotic expansion
\begin{displaymath}
 	\adeg H^{0}(\pX,\aL^{\otimes k}\otimes\aN)_{L^2}=h_{\ov{\aL}}(\pX)\frac{k^{d}}{d!}+o(k^{d})\text{ as }k\to+\infty.
\end{displaymath}
\end{theorem}
We recall that the arakelovian heights appearing in the previous theorem enjoy the usual Northcott's type finiteness properties \cite{GFM}.

The proof exploits recent tools in pluropotential thoery. The connection between pluripotential theory and arithmetic intersection theory is already present in work of Bost \cite{Bost} and Berman-Boucksom \cite{Berman-Boucksom}. In this last reference, the relation between the energy functional, definig the class of finite energy metrics, and secondary Bott-Chern forms and heights, is exploited. As shown in \cite{Berman-Boucksom} the energy functional can be linked to Donaldson type $\mathcal{L}_k$ functionals, that in the arithmetic setting can be expressed in terms of the covolume of the lattices of $k$-th powers of an hermitian line bundle. While \emph{loc. cit.} deals with semi-positive metrics with minimal singularities (which for a semi-ample line bundle means that they are locally bounded), our point is to deal with arbitrary metrics of finite energy. Specially relevant for our purposes is Berndtsson's positivity theorem \cite{Berndtsson}, that is used to prove convexity properties of the $\mathcal{L}_k$ functionals along geodesic paths joining metrics with minimal singularities. This convexity result is then combined with the results of the first author on Bergman kernels for semi-positive (but possibly non-ample) hermitian line bundles \cite{Berman:local-morse}. 

Natural examples of application of theorems \ref{thm:main}--\ref{thm:main_cor} are provided by integral models of log-canonical sheaves $K_{X}(D)$ attached to couples $(X,D)$ formed by a smooth projective variety $X$ over $\QQ$ together with a divisor with normal crossings $D\subset X$. In these situations, under the assumption that $K_{X}(D)$ be semi-ample and big, there is a K\"ahler-Einstein metric of finite energy on $K_{X}(D)_{\CC}$.  The content of the main theorem is then the asymptotic behaviour of the covolumes (with respect to the $L^2$ norms) of the spaces of suitable integral cusp forms. This setting applies, in particular, to arithmetic models of non-compact Shimura varieties.  In section \ref{sec:Hilbert} we will focus on the case of Hilbert modular surfaces, establishing an arithmetic analogue of the classical theorem computing the dimension of the space of cusp forms in terms of a special value of a Dedeking zeta function \cite[Chap. IV, Thm 1.1 and Thm. 4.2]{vdG}. Let $F$ be a real quadratic number field of discriminant $D\equiv 1 \mod 4$. For a sufficiently divisible integer $\ell\geq 1$, we construct some naive arithmetic model of a toroidal compactification of the Hilbert modular surface attached to $\Gamma_{F}(\ell)\subset\SL_{2}(\OO_F)$. Here $\Gamma_{F}(\ell)$ is the principal congruence subgroup of level $\ell$ of $\SL_{2}(\OO_F)$. Our model $\ov{\mathcal{H}}(\ell)$ is projective over $\Spec\Int$ and smooth over the generic fiber, and comes with a natural semi-ample model $\omega$ of the sheaf of Hilbert modular forms of parallel weight $2$. Furthermore, $\omega_{\QQ}$ is big and there is a Kodaira-Spencer isomorphism $\omega_{\QQ}\simeq K_{\ov{\mathcal{H}}(\ell)_{\QQ}}(D)$, where $D$ is the boundary divisor of the toroidal compactification over $\QQ$. Hence, $\omega_{\CC}$ can be equiped with the K\"ahler-Einstein metric which coincides, through the Kodaira-Spencer isomorphism, with the usual pointwise Petersson metric. It is actually known to be log-singular. Now, if $\aK$ is any model of the canonical sheaf, the space of global sections $H^{0}(\ov{\mathcal{H}}(\ell),\omega^{\otimes k}\otimes \aK)$ is an integral structure in the space of Hilbert cusp forms of level $\ell$ and parallel weight $2k+2$. Its $L^2$ metric is the Petersson pairing, that we shall indicate by $\pet$. An application of Theorem \ref{thm:main} (or Theorem \ref{thm:main_cor}) and the results of Bruinier-Burgos-K\"uhn \cite[Thm. 6.4]{BBK} yields the following statement (Theorem \ref{thm:Hilbert}).
\begin{theorem}\label{thm:Hilbert_intro}
The following asymptotic formula holds:
\begin{displaymath}
	\begin{split}
	\adeg H^{0}(\ov{\mathcal{H}}(\ell), &\,\omega^{\otimes k}\otimes\aK)_{\pet}=\\
	&-\frac{k^{3}}{6}d_{\ell}\zeta_{F}(-1)\left(\frac{\zeta_{F}^{\prime}(-1)}{\zeta_{F}(-1)}+\frac{\zeta^{\prime}(-1)}{\zeta(-1)}+\frac{3}{2}+
	\frac{1}{2}\log D\right) +o(k^{3}),
	\end{split}
\end{displaymath}
where $d_{\ell}=[\QQ(\zeta_{\ell}):\QQ][\Gamma_{F}(1):\Gamma_{F}(\ell)]$, and $\zeta_{F}$ is the Dedekind zeta function of $F$.
\end{theorem}
Observe that the result does not depend on the particular model $\aK$ of the canonical sheaf, nor on the models of $\ov{\mathcal{H}}(\ell)$ and $\omega$ over $\Spec\Int$. A more intrinsic formulation would involve only the spaces of integral sections of finite $L^{2}$ norm $H^{0}(\ov{\mathcal{H}}(\ell), \,\omega^{\otimes k+1})_{\pet, L^{2}}$, with respect to the Petersson norm. Unfortunately, the lack of a good integral model of the toroidal compactification prevents us from doing this. Roughly speaking, the difficulty stems from the fact that we ignore if the flat closure of the boundary of $\ov{\mathcal{H}}(\ell)_{\QQ}$ is a Cartier divisor in our naive model over $\Spec\Int$.

Let us finish this introduction by some words about possible applications. The classical arithmetic Hilbert-Samuel theorem for positive smooth hermitian line bundles is usually applied to produce global integral sections with small sup norm. This requires a comparison of $L^{2}$ and sup norms, which is usually proved by Gromov's inequality in the Arakelov geometry litterature. In fact, the distorsion between the norms coincides with the sup-norm of the corresponding Bergman function and all that is needed for the comparison is a bound which is sub-exponential in the power $k$ of the line bundle (this is called the Bernstein-Markov property in the pluripotential litterature). In particular it applies to any continuous metric. However, in the full generality of semi-positive metrics of finite energy, such bounds are not available. Indeed, the sup norm is in general not even finite. However, the question makes sense for log-singular metrics if one considers the subspace of sections vanishing along the divisor of singularities. This is the case for cusp forms on Hilbert modular surfaces, as considered above. But the only examples we know of with sub-exponential distortion is given by cusp forms on modular curves \cite{Parson}, \cite{Xia}. Already this example shows the interest and non-triviality of the question, since the proof invokes the Ramanujan's bounds for the Fourier coefficients of modular forms, proven by Deligne as a consequence of the Weil conjectures. We hope to come back to this topic in the future.

We briefly review the contents of the article. In section \ref{sec:pluripotential} we recall some definitions and facts we need on pluripotenital theory and metrics of finite energy. We give an easy version of the local holomorphic Morse inequalities proven by the first author. In section \ref{sec:geodesics} we recall the notion of bounded geodesic paths between semi-positive hermitian metrics with minimal singularities, and study the differential of the energy functional along such paths, as well as convexity of $\mathcal{L}_k$ functionals through Berndtsson's positivity theorem. The core of the article is Theorem \ref{thm:asymptotics-L-functional} on convergence of $\mathcal{L}_k$ functionals towards energy functionals. In section \ref{sec:heights} we introduce the height with respect to a semi-positive line bundle of finite energy, and we prove Theorem \ref{thm:main}. We also discuss the case of log-singular hermitian line bundles, with singularities along rational normal crossing divisors (Theorem \ref{thm:main_cor}). We conclude with section \ref{sec:Hilbert}, where we construct some naive integral models of toroidal compactifications of Hilbert modular schemes and prove Theorem \ref{thm:Hilbert_intro}. 

\section{Pluripotential theory and Bergman measures}\label{sec:pluripotential}
In this section we recall some basics on pluripotential theory following \cite{BEGZ}, as well as the construction of the $\pL_{k}$ functionals of \cite{Berman-Boucksom}. We discuss their relation with the Monge-Amp\`ere and the Bergman measures.

\subs Let $X$ be a compact complex manifold of dimension $n$ and $L$ a holomorphic line bundle on $X$. We assume that $L$ admits a smooth semi-positive hermitian metric $h_{0}$, whose first Chern form $\c1(L,h_{0})$ we denote $\omega_{0}$. This metric will be fixed once and for all. We introduce the space of $\omega_{0}$-psh functions, namely the convex subset of $L^{1}(X,\RR)$
\begin{displaymath}
	\PSH(X,\omega_{0})=\lbrace \phi\in L^{1}(X,\RR)\text{ u.s.c. on $X$}\mid \omega_{\phi}:=\omega_{0}+dd^{c}\phi\geq 0\rbrace.
\end{displaymath} 
Here $dd^{c}\phi$ is to be interpreted in the sense of currents. Observe we can identify $\PSH(X,\omega_{0})$ with the space of singular semi-positive metrics on $L$, by the rule $\phi\mapsto h_{\phi}:=h_{0}e^{-\phi}$. The space of bounded $\omega_{0}$-psh functions $\PSH(X,\omega_0)\cap L^{\infty}(X)$ is also called the space of $\omega_0$-psh functions with minimal singularities. We will consider several operators defined on subspaces of $\PSH(X,\omega_{0})$.

The positivity assumption on $L$ in particular guarantees that it is \emph{nef}. We will assume throughout that $\vol(L)>0$ or, equivalently, that $L$ is \emph{big}.

\subsection{The Monge-Amp\`ere operator} By the work of Bedford-Taylor \cite{BT}, the operator sending a smooth $\omega_0$-psh function $\phi$ to the semi-positive differential form $(\omega_{0}+dd^{c}\phi)^{n}$ can be extended to $\PSH(X,\omega_{0})\cap L^{\infty}(X)$. 

For our purposes it will be convenient to introduce the following normalized Monge-Amp\`ere operator:
\begin{displaymath}
	\MA(\phi):=\frac{1}{\vol(L)}\omega_{\phi}^{n},
\end{displaymath}
This is a measure putting no mass on pluripolar sets, of total mass at most $1$.
\subsection{The Aubin-Mabuchi energy functional}\label{subsec:finite-energy} We define the energy functional on $\PSH(X,\omega_{0})\cap L^{\infty}(X)$ by the formula
\begin{equation}\label{eq:1}
	\mathcal{E}(\phi):=\frac{1}{(n+1)\vol(L)}\sum_{j=0}^{n}\int_{X}\phi\omega_{\phi}^{j}\wedge\omega_{0}^{n-j}\quad (\in\RR).
\end{equation}
The following proposition summarizes some of its main properties we will need.
\begin{proposition}\label{prop:1}
i. The functional $\pE$ is non-decreasing and continuous in the following cases:
\begin{itemize}
	\item along point-wise decreasing sequences of bounded $\omega_0$-psh functions;
	\item along uniformly convergent sequences of bounded $\omega_0$-psh functions.
\end{itemize}

ii. The G\^ateaux derivative of $\pE$ on $\PSH(X,\omega_{0})\cap L^{\infty}(X)$ is given by
\begin{displaymath}
	d\pE_{\mid\phi}=\MA(\phi).
\end{displaymath}

iii. With respect to the convex structure of $\PSH(X,\omega_{0})\cap L^{\infty}(X)$, $\pE$ is concave. Namely, for given $\omega_0$-psh functions with minimal singularities $\phi_{0}, \phi_{1}$, the function $t\mapsto\pE((1-t)\phi_{0}+t\phi_{1})$ is concave on $[0,1]$.
\end{proposition}
\begin{proof}
We refer to \cite[Prop. 2.10]{BEGZ}, \cite[Prop. 4.3--4.4]{Berman-Boucksom}.
\end{proof}
The preceding statement suggests a unique monotone and uppersemi-continuos extension of $\pE$ to $\PSH(X,\omega_{0})$ by the rule
\begin{displaymath}
	\pE(\phi)=\inf_{\psi\geq\phi}\pE(\psi)\in [-\infty,+\infty[,
\end{displaymath}
where the $\inf$ runs over bounded $\omega_0$-psh functions dominating $\phi$. It is clear that it remains non-decreasing and concave. Following \cite[Def. 2.9]{BEGZ}, the space of \emph{finite energy} functions in $\PSH(X,\omega_{0})$ is defined as
\begin{displaymath}
	\pE^{1}(X,\omega_{0})=\lbrace \phi\in\PSH(X,\omega_{0})\mid \pE(\phi)>-\infty\rbrace.
\end{displaymath}
Functions of finite energy have full Monge-Amp\`ere mass \cite[Prop. 2.11]{BEGZ}. Moreover, the operator $\pE$ extended to $\pE^{1}(X,\omega_{0})$ is still continuous along decreasing sequences and formula \eqref{eq:1} still holds on this space, when the exterior products are interpreted as non-pluripolar products of currents \cite[Prop. 2.10--2.17, Cor. 2.18]{BEGZ}. 

\subsection{Determinant of the cohomology and the functionals $\pL_{k}$}\label{subsection:det-coh} Let $\phi$ be a $\omega_0$-psh function of finite energy. For every $k\geq 0$, the cohomology group $H^{0}(X,L^{\otimes k}\otimes K_{X})$ is endowed with a $L^2$ metric. Indeed, given sections $s_1, s_2$, we write in local coordinates $s_{i}(z)=f_{i}(z)\ell_{i}(z)^{k}dz_{1}\wedge\ldots\wedge dz_{n}$, where the $f_{i}(z)$ are holomorphic functions and the $\ell_{i}$ are local holomorphic sections of $L$. Then we put
\begin{displaymath}
	\langle s_{1},s_{2}\rangle_{k\phi}(z):=f_{1}(z)\ov{f_{2}(z)}h_{\phi}(\ell_{1}(z),\ell_{2}(z))^{k}|dz_{1}\wedge\ldots\wedge dz_{n}|^{2},
\end{displaymath}
where we follow the notation
\begin{displaymath}
	|dz_{1}\wedge\ldots\wedge dz_{n}|^{2}=i^{n} dz_{1}\wedge d\cz_{1}\wedge\ldots\wedge dz_{n}\wedge d\cz_{n}.
\end{displaymath}
By the finite energy condition, the form $\langle s_{1},s_{2}\rangle_{k\phi}$ is an integrable top differential form and we may thus define
\begin{displaymath}
	\langle s_{1},s_{2}\rangle_{k\phi,X}=\int_{X}\langle s_{1},s_{2}\rangle_{k\phi}(z).
\end{displaymath}
The determinant line bundle $\det H^{0}(X,L^{\otimes k}\otimes K_{X})=\bigwedge^{N_{k}} H^{0}(X,L^{\otimes k}\otimes K_{X})$ (with $N_{k}=\dim H^{0}(X,L^{\otimes k}\otimes K_{X})$) inherits a $L^{2}$ metric as well, for which we use the same notation. In other words, given a basis $s_{k}$ of $H^{0}(X,L^{\otimes k}\otimes K_{X})$ we put
\begin{displaymath}
	\langle s_{1}\wedge\ldots\wedge s_{N_k}, s_{1}\wedge\ldots\wedge s_{N_k}\rangle_{k\phi,X}:=\det(\langle s_{i},s_{j}\rangle_{k\phi,X}).
\end{displaymath}
The functional $\pL_{k}$ is defined on $\pE^{1}(X,\omega_{0})$ by
\begin{equation}\label{eq:2}
	\pL_{k}(\phi):=-\frac{1}{k N_{k}}\log\det(\langle s_{i}^{(k)},s_{j}^{(k)}\rangle_{k\phi,X}),
\end{equation} 
where the basis $(s_{j}^{(k)})$ is orthonormal with respect to $\langle\cdot,\cdot\rangle_{0}$, namely $h_{0}$. Observe the definition does not depend on the  choice of orthonormal basis.
\subsection{The Bergman measure} 
As for the operator $\pE$ we will be interested in the G\^ateaux derivative of the functional $\pL_{k}$. It will be expressed in terms of the Bergman measure, that we proceed to recall. Let $\phi$ be of finite energy, and take an orthonormal basis $(t_{j}^{(k)})$ of $H^{0}(X,L^{\otimes k}\otimes K_{X})$, with respect to $h_{\phi}$. Then, according to the notations in \textsection 1.4, the following expression defines a probability measure:
\begin{equation}\label{eq:3}
	\beta_{k\phi}(x):=\frac{1}{N_{k}}\sum_{j=1}^{k}\langle t_{i}^{(k)},t_{i}^{(k)}\rangle_{k\phi}(x).
\end{equation}
We call it the \emph{Bergman measure}. It acts on measurable functions by integration. We are now in position to compute the derivative of $\pL_{k}$.
\begin{proposition}\label{prop:2}
The G\^ateaux derivative of $\pL_{k}$ at $\phi\in\PSH(X,\omega_{0})\cap L^{\infty}(X)$ is given by
\begin{displaymath}
	d\pL_{k\mid\phi}=\beta_{k\phi}.
\end{displaymath}
\end{proposition}
\begin{proof}
The proof goes as in \cite[Lemma 5.1]{Berman-Boucksom}.
\end{proof}
\begin{remark}
The variational formulas stated in propositions \ref{prop:1}--\ref{prop:2} can be generalized to $\pE^{1}(X,\omega_{0})$. This is achieved by approximation by functions with minimal singularities. However, in the present article we won't need such level of generality.
\end{remark}
The next statement provides an inequality between the measures $\MA(\phi)$ and $\beta_{k\phi}$. It consists in the easy version of the more general local Morse inequalities proven by the first author \cite{Berman:local-morse}. By the variational formulas and some convexity properties to be stated in the next section, this will yield a comparison of the operators $\pE$ and $\pL_{k}$.
\begin{proposition}{(Local Morse inequalities)}\label{prop:3}
Let $\phi$ be a smooth $\omega_{0}$-psh function. Then there exists a sequence of positive numbers $\delta_{k}\to 0$ such that
\begin{displaymath}
	\beta_{k\phi}\leq (1+\delta_{k})\MA(\phi).
\end{displaymath}
\end{proposition}
\begin{proof}
We refer to \cite[Thm. 1.1]{Berman:local-morse}. We give a sketch of the argument. The proof exploits the sub-mean inequality of holomorphic functions and the extremal property of the Bergman kernel. Indeed, recall that
\begin{displaymath}
	N_{k}\beta_{k\phi}(x)=\sup_{s\in H^{0}(X,L^{\otimes k}\otimes K_{X})}\frac{\langle s,s\rangle_{k\phi}(x)}{\langle s,s\rangle_{k\phi,X}}.
\end{displaymath}
In a coordinate neighborhood $V$ centered at $x$, we trivialize $L$ and write $\langle s,s\rangle_{k\phi}(z)=|f(z)|^{2}e^{-k\Phi(z)}$, where $f$ is the holomorphic function corresponding to $s$ under the trivialization and $\Phi$ is a smooth function on $V$ with $dd^{c}\Phi=\omega_{\phi}$. We may choose the trivialization so that
\begin{equation}\label{eq:4}
	\Phi(z)=\Phi_{0}(z)+o(|z|^{2}),\quad \Phi_{0}(z)=\sum_{i}\lambda_{i}|z_{i}|^{2}.
\end{equation}
Here the $\lambda_{i}$ are the eigenvalues of the Hessian of $\Phi$ at $z=0$. Now we apply the sub-mean inequality to $|f(z)|^{2}$ on tori centered at $z=0$ and of radius $r$, and integrate in $r\in [0,R_{k}]$ for some small radius $R_{k}$. We get
\begin{displaymath}
	|f(0)|^{2}\prod_{j}\left(\int_{0}^{R_{k}}e^{-\lambda_{j}r^{2}}d(r^{2})\right)\leq\int_{|z|\leq R_{k}}|f(z)|^{2}e^{-k\Phi_{0}(z)}
	|dz_{1}\wedge\ldots\wedge dz_{n}|^{2}.
\end{displaymath}
If we put $R_{k}=(\log k)/k^{1/2}$ and take \eqref{eq:4} into account, we can replace $\Phi_{0}$ by $\Phi$ in the previous inequality, up to introducing a factor $(1+\delta_{k})$, with $\delta_{k}\to 0$. Finally, one computes the Gaussian integral to derive
\begin{displaymath}
	\langle s,s\rangle_{k\phi}(x)\leq (1+\delta_{k})\MA(\phi)_{x}\int_{B_{R_k}}\langle s,s\rangle_{k\phi}(z).
\end{displaymath}
This concludes the proof.
\end{proof}

\section{Geodesics in the space of metrics}\label{sec:geodesics}
\subs Following ideas of Mabuchi, Semmes and Donaldson, we consider the problem of joining psh functions in $\PSH(X,\omega_{0})$ by \emph{geodesic} paths. The issue of the existence, uniqueness and regularity of geodesics in $\PSH(X,\omega_{0})$ for a suitable riemannian structure  is a delicate one.\footnote{In the literature, $\omega_{0}$ is usually a K\"ahler form and one restricts to the space of K\"ahler potentials, namely those smooth $\phi$ such that $\omega_{0}+dd^{c}\phi>0$.} The present article deals with a weak notion of geodesic, namely bounded geodesics, good enough to derive consequences for the operators $\pE$ and $\pL_{k}$. We anticipate the main features: the energy operator $\pE$ is affine along geodesics, while the operator $\pL_{k}$ is convex. We stress that the last convexity property relies on positivity results for direct images of line bundles in adjoint form, due to Berndtsson \cite{Berndtsson}. The relevant study of bounded geodesics was very recently undertaken in Berman-Berndtsson \cite{Berman-Berndtsson}, but as a courtesy to the reader we have recalled the proofs of the main properties that we will need.

\subsection{Subgeodesics} Let $\phi_{0}$ and $\phi_{1}$ be in $\PSH(X,\omega_{0})$. Let $\pT$ be the Riemann surface with boundary $\pT=[0,1]+i\RR$ and $\pi_{X},\pi_{\pT}$ the natural projections from $X\times\pT$. With the functions $\phi_{0}$ and $\phi_{1}$ we construct in an obvious manner an $i\RR$ invariant function $\phi_{\mid\pd(X\times\pT)}$ on the boundary $\pd(X\times\pT)$. The space of subgeodesics between $\phi_{0}$ and $\phi_{1}$ is by definition
\begin{displaymath}
	\pK:=\lbrace \psi\text{ u.s.c. on } X\times\pT,\psi\in \PSH(\pi^{\ast}\omega_{0},X\times\overset{\circ}{\pT})\text{ and }\psi_{\mid\pd(X\times\pT)}\leq \phi_{\mid\pd(X\times\pT)}\rbrace.
\end{displaymath}
For a subgeodesic $\psi$, we will usually write $\psi_{t}(x)$ instead of $\psi(x,t)$.  
\begin{proposition}\label{prop:4}
Let $\psi$ be a subgeodesic between $\phi_{0}$ and $\phi_{1}$.

i. The function $t\mapsto\pE(\psi_{t})$ is $psh$ on $(0,1)+i\RR$. Moreover, if $\psi$ is locally bounded, then
\begin{displaymath}
	d_{t}d_{t}^{c}\pE(\psi_{t})=\pi_{\pT\ast}((\pi_{X}^{\ast}\omega_{0}+dd^{c}\psi)^{n+1}).
\end{displaymath}

Assume moreover that $\psi_{t}$ does not depend on the imaginary part of $t$. Then:

ii. as a function of $t\in (0,1)$, $t\mapsto\psi_{t}$ is convex;

iii. if $\psi_{t}$ is bounded on $X$ uniformly in $t\in[0,1]$ and continuous (as a function of $t$) at $t=0,1$, then the right derivative of $\pE(\psi_{t})$ at $0$ satisfies
\begin{displaymath}
	\frac{d}{dt}\Big |_{t=0^{+}}\pE(\psi_{t})\leq\int_{X} (\frac{d}{dt}\Big |_{t=0^{+}}\psi_{t})\MA(\psi_{0}).
\end{displaymath}
\end{proposition}

\begin{proof}[Proof of Proposition \ref{prop:4}]
For the first item, if $\psi$ is locally bounded the curvature equation is easily checked. In particular it implies the psh property. The general case follows by approximation by bounded functions. We refer to \cite[Prop. 6.2]{BBGZ} for further details.

Assume, until the end of the proof, that $\psi_{t}$ is independent of the imaginary part of $t$.

For the second item, the function $u(t)$ is convex on $(0,1)$ by subharmonicity on $(0,1)+i\RR$ and independence of $\text{Im}(t)$.

For the last property, we first recall that $\pE$ is concave and that $d\pE_{\mid\phi}=\MA(\phi)$ on $\PSH(X,\omega_{0})\cap L^{\infty}(X)$. Therefore, we have for $t>0$
\begin{displaymath}
	\frac{\pE(\psi_{t})-\pE(\psi_{0})}{t}\leq \int_{X}\frac{\psi_{t}-\psi_{0}}{t}\MA(\psi_{0}).
\end{displaymath}
Now the function $\psi_{t}$ is convex in $t\in[0,1]$, because it is convex on $(0,1)$ by \emph{ii} and continuous at $t=0,1$. It follows that for $0<t\leq 1$, $(\psi_{t}-\psi_{0})/t$ is uniformly bounded above by $\sup_{X}(\psi_{1}-\psi_{0})<+\infty$. Besides, it is decreasing as $t\searrow 0$, by convexity again. Therefore, by the monotone convergence theorem we conclude
\begin{displaymath}
	\frac{d}{dt}\Big |_{t=0^{+}}\pE(\psi_{t})\leq\int_{X}(\frac{d}{dt}\Big |_{t=0^{+}}\psi_{t})\MA(\psi_{0}),
\end{displaymath}
as was to be shown.
\end{proof}
\subsection{Bounded geodesics}
We maintain the notations of the preceding section. Let $\phi_0,\phi_{1}$ be bounded $\omega_0$-psh function and $\psi$ a subgeodesic between them. We say that $\psi$ is a \emph{bounded geodesic} joining $\phi_{0}$ and $\phi_{1}$ if the following conditions are fulfilled:
\begin{itemize}
	\item $\psi_{t}$ is independent of the imaginary part of $t$;
	\item $\psi_{t}$ is bounded on $X$ uniformly in $t\in [0,1]$, and converges uniformly to $\phi_{0}$ (resp. $\phi_{1}$) as $t\to 0$ (resp. $t\to 1$).
	\item $\psi$ solves the degenerate Monge-Amp\`ere equation on $X\times\overset{\circ}{\pT}$
		\begin{equation}\label{eq:8}
			(\pi^{\ast}\omega_{0}+dd^{c}\psi)^{n+1}=0.
		\end{equation}
\end{itemize}
Observe that the boundedness assumption permits to state \eqref{eq:8} within the frame of Bedford-Taylor theory. 
\begin{proposition}\label{prop:5}
Let $\phi_{t}$ be a geodesic path between $\phi_0, \phi_{1}\in\PSH(X,\omega_{0})\cap L^{\infty}(X)$. Then:

i. the function $t\mapsto\pE(\phi_{t})$ is affine on $[0,1]$;

ii. the following inequality holds: 
\begin{displaymath}
	\pE(\phi_{1})-\pE(\phi_{0})\leq\int_{X}(\frac{d}{dt}\Big |_{t=0^{+}}\phi_{t})\MA(\phi_{0}).
\end{displaymath}
\end{proposition}
\begin{proof}
First of all, by Proposition \ref{prop:4} \emph{i} and equation \eqref{eq:8}, the function $\pE(\phi_{t})$ is harmonic on $(0,1)+i\RR$, and independent of $\text{Im}(t)$ by assumptions on $\phi_{t}$. Therefore it is affine on $(0,1)$. Moreover, $\phi_{t}$ uniformly converges to $\phi_{0}$ (resp. $\phi_1$) as $t\to 0$ (resp. $t\to 1$). Then, by Proposition \ref{prop:1} \emph{ii} we deduce that $\pE(\phi_{t})$ is continuous at $t=0,1$. This shows the first assertion. For the second one, we apply Proposition \ref{prop:4} \emph{iii} and take into account the previous afineness property, that guarantees
\begin{displaymath}
	\frac{d}{dt}\Big |_{t=0^{+}}\pE(\phi_{t})=\pE(\phi_{1})-\pE(\phi_{0}).
\end{displaymath}
The proof is complete.
\end{proof}
\begin{proposition}\label{prop:6}
Let $\phi_0,\phi_{1}\in\PSH(X,\omega_{0})\cap L^{\infty}(X)$. Then there exists a bounded geodesic $\phi_{t}$ between $\phi_{0}$ and $\phi_{1}$.
\end{proposition}
\begin{proof}
It will be convenient to introduce the annulus $A=\lbrace 1\leq |z|\leq e\rbrace$, so that $z\to e^{z}$ defines a locally conformal mapping from $\pT$ to $A$. We have the corresponding notion of (sub)geodesics. Radial functions [resp. (sub) geodesics] on $A$ pull-back to functions on $\pT$ independent of $\text{Im}(t)$ [resp. (sub) geodesics]. We may thus work on $X\times A$. We consider the upper envelope
\begin{displaymath}
	\phi:=\sup\lbrace\psi\text{ subgeodesic on } X\times A \text{ between }\phi_{0},\phi_{1}\rbrace.
\end{displaymath} 
Observe the set of subgeodesics under consideration is not empty (for an example, see the barrier below). First we claim that $\phi$ is a radial subgeodesic. It is a $\omega_{0}$-psh function on $X\times\overset{\circ}{A}$: indeed, the upper semi-continuous regularization $\phi^{\ast}$ is a candidate in the $\sup$, hence $\phi=\phi^{\ast}$. In addition, the radial function
\begin{displaymath}
	\widetilde{\phi}_{z}(x)=\sup_{\theta\in [0,2\pi]}\phi_{ze^{i\theta}}(x)
\end{displaymath}
is also a candidate in the $\sup$, so that $\phi=\widetilde{\phi}$ is radial. Secondly, we claim that $\phi$ is bounded and uniformly converges to $\phi_{0}$ (resp. $\phi_{1}$) when $|z|\to 1$ (resp. $|z|\to e$). For this, we follow \cite[Sec. 2.2]{Berndtsson:Bando-Mabuchi} and introduce a barrier:
\begin{displaymath}
	\chi_{z}=\max(\phi_{0}-C\log|z|,\phi_{1}+C(\log|z|-1)).
\end{displaymath}
For $C$ sufficiently large and because $\phi_0,\phi_{1}$ are bounded, this barrier is a candidate in the $\sup$. By \emph{loc. cit.} one has
\begin{displaymath}
	\phi_{0}-C\log|z|\leq \phi_{z}\leq\phi_{0}+C\log|z|
\end{displaymath}
and similarly for $\phi_{1}$. These inequalities show that $\phi_{z}$ is uniformly bounded in $z$ and uniformly converges to $\phi_{0},\phi_{1}$ when $|z|\to 1,e$. 

To conclude, it remains to show that $\phi$ so defined satisfies the degenerate Monge-Amp\`ere equation \eqref{eq:8} on $X\times\overset{\circ}{A}$. The argument is standard and based on the classical Perron method. We provide the details for the sake of completeness (see also \cite[proof of Prop. 2.10]{Berman-Boucksom}). Let $B$ be an open ball, relatively compact in $X\times\overset{\circ}{A}$. We already saw that the function $\phi$ is bounded, in particular on $\ov{B}$. Therefore, by Bedford-Taylor theory \cite[Thm. D]{BT}\footnote{Strictly speaking, \emph{loc. cit.} requires a continuous boundary datum. The bounded case follows by approximation by a decreasing sequence of continuous functions (possible by upper semi-continuity of $\phi$), by the minimum principle \cite[Thm. A]{BT} and the continuity of Monge-Amp�re measures along decreasing sequences of psh functions.} we can find a $\pi^{\ast}_{X}\omega_{0\mid B}$-psh function $\psi$ on $B$, which coincides with $\phi$ on $\pd B$. Then we define
\begin{displaymath}
	\widetilde{\phi}=
	\begin{cases}
		\psi\quad\text{on }\ov{B},\\
		\phi\quad\text{on }(X\times A)\setminus\ov{B}.
	\end{cases}
\end{displaymath}
On the one hand, $\widetilde{\phi}\geq \phi$. Indeed, $\psi$ is a decreasing limit of Perron envelopes on $B$ with boundary values decreasing to $\phi$, thus $\psi\geq\phi_{\mid\ov{B}}$. On the other hand, $\widetilde{\phi}$ is still a psh function in the $\sup$ defining $\phi$. Therefore $\phi=\widetilde{\phi}$. Hence $\phi$ satisfies the degenerate Monge-Amp\`ere equation on $B$. The proof is now complete.
\end{proof}
\subsection{Berndtsson's positivity and convexity of $\pL_{k}$. Consequences}
In the previous sections we studied the behavior of the functional $\pE$ along subgeodesics. We now consider the operator $\pL_{k}$. Together with the variational formulas of $\pE$ and $\pL_{k}$, we are going to establish a comparison between both, at least for big values of $k$.

Let the notations be as before. Let $\phi_{t}$ be a subgeodesic between functions $\phi_{0},\phi_{1}$ with minimal singularities. We assume that $\phi_{t}$ is uniformly bounded. This is for instance the case for geodesic paths. For every $t$, $\phi_{t}$ defines a semi-positively curved singular bounded metric on $L$, $h_{t}=h_{0}e^{-\phi_t}$. Then $\det H^{0}(X,L^{\otimes k}\otimes K_{X})$ inherits a $L^2$ hermitian structure that, we recall, we denote $\langle\cdot,\cdot\rangle_{k\phi_t,X}$. This family of metrics glues into a singular hermitian metric on the constant sheaf over $\pT$ of fiber $\det H^{0}(X,L^{\otimes k}\otimes K_{X})$.\footnote{The resulting family is locally integrable in $t$ by Fubini's theorem and the local integrability of $\phi_{t}\in\PSH(X\times\overset{\circ}{\pT},\pi^{\ast}_{X}\omega_{0})\subset L^{1}_{\text{loc}}(X\times\overset{\circ}{\pT},\RR)$} 

The previous construction can be equivalently seen as the family $L^{2}$ metric on $\det \pi_{\pT\ast}(\pi^{\ast}_{X}(L^{\otimes k})\otimes K_{X\times\pT/\pT})$ attached to the singular semi-positive metric $\pi^{\ast}_{X}(h_{0})e^{-\phi_{t}}$ on $\pi^{\ast}_{X}(L)$. Observe that $\pi_{\pT}$ is a proper submersion. This suggests to use positivity properties for direct images of hermitian line bundles. The main result we need is due to Berndtsson \cite{Berndtsson} in the smooth case and generalized to singular metrics by Berndtsson-P\u{a}un \cite{BerndtssonPaun}.

\begin{proposition}\label{prop:7}
Let $\phi_{t}$ be a uniformly bounded subgeodesic between $\phi_0$ and $\phi_1$. 

i. The function $t\mapsto\pL_{k}(\phi_{t})$ is psh on $(0,1)+i\RR$. 

ii. If $\phi_{t}$ is a geodesic, then $t\mapsto\pL_{k}(\phi_{t})$ is convex on $[0,1]$.
\end{proposition}

\begin{proof}
In the case $\phi_t$ is a smooth subgeodesic, the first property follows from \cite[Thm. 1.1]{Berndtsson}. In the general case, the elaboration on the case of singular metrics \cite[Thm 3.5]{BerndtssonPaun} can be applied. For the second property, convexity on $(0,1)$ is obvious since the geodesic conditions ensure $\phi_t$ does not depend on $\text{Im}(t)$. Also, $\phi_{t}$ uniformly converges to $\phi_0$ (resp. $\phi_{1}$) as $t\to 0$ (resp. $t\to 1$). By the dominate convergence theorem, $\pL_{k}(\phi_t)$ is continuous at $t=0,1$. This concludes the proof.
\end{proof}

We are now in position to state and prove the main theorem of the section.

\begin{theorem}\label{thm:asymptotics-L-functional}
Let $\phi\in\pE^{1}(X,\omega_{0})$. Then we have
\begin{displaymath}
	\lim_{k\to+\infty}\pL_{k}(\phi)=\pE(\phi).
\end{displaymath}
\end{theorem}
\begin{proof}
We begin by showing the inequality
\begin{equation}\label{eq:9}
	\limsup_{k\to +\infty}\pL_{k}(\phi)\leq\pE(\phi).
\end{equation}
Let us consider a decreasing sequence of smooth functions $\phi_{j}\searrow\phi$, which always exists. For every $j$, the quantity $\pL_{k}(\phi_{j})$ is well-defined and clearly satisfies $\pL_{k}(\phi)\leq\pL_{k}(\phi_{j})$. By \cite[Thm. A]{Berman-Boucksom}\footnote{Strictly speaking, \emph{loc. cit} applies to $L^{\otimes k}$. The arguments can be easily adapted to obtain the corresponding results for $L^{\otimes k}\otimes K_{X}$.}
\begin{displaymath}
	\lim_{k}\pL_{k}(\phi_{j})=\pE(P_{X}\phi_{j}),
\end{displaymath}
where $P_{X}\phi_{j}$ is the $\omega_0$-psh projection of $\phi_{j}$, namely
\begin{displaymath}
	P_{X}\phi_{j}=\sup\lbrace\psi\in\PSH(X,\omega_{0})\mid\psi\leq\phi_{j}\rbrace.
\end{displaymath}
By the very definition of this projection, we have the inequalities  $\phi\leq P_{X}\phi_{j}\leq\phi_{j}$ and the sequence $P_{X}\phi_{j}$ is decreasing. Therefore, $P_{X}\phi_{j}$ decreases to $\phi$ and by the monotonicity of the functional $\pE$ we have 
\begin{displaymath}
	\lim_{j}\pE(P_{X}\phi_{j})=\pE(\phi).
\end{displaymath}
This concludes the proof of \eqref{eq:9}.

Now for the inequality
\begin{equation}\label{eq:10}
	\liminf_{k\to+\infty}\pL_{k}(\phi)\geq\pE(\phi).
\end{equation}
Let us introduce the functional on $\pE^{1}(X,\omega_{0})$
\begin{displaymath}
	\pF_{k}(\psi):=\pL_{k}(\psi)-\pE(\psi).
\end{displaymath}
Observe that $\pF_{k}$ is continuous along decreasing sequences, because $\pL_{k}$ and $\pE$ are. Indeed, for $\pL_{k}$ this follows from the dominate convergence theorem, while for $\pE$ this property was already invoked previously (Proposition \ref{prop:1} \emph{i}). Also, it is easily seen to be invariant under translation of $\psi$ by constants. We claim that the inequality
\begin{equation}\label{eq:11}
	\pF_{k}(\phi)\geq \delta_{k}\pE(\phi)
\end{equation}
holds, with $\delta_{k}$ being the sequence of Proposition \ref{prop:3} applied to $\phi_{0}:=0$. This will be enough to conclude. For this, let $\psi_{j}$ be a sequence of $\omega_0$-psh functions with minimal singularities decreasing to $\phi$. Let us fix the index $j$. After possibly making a translation by a constant, we may assume that $\psi_{j}\leq 0$. Let $\phi_{t}$ be a bounded geodesic between $\phi_{0}=0$ and $\phi_{1}:=\psi_{j}$ (Proposition \ref{prop:6}). The function $\pF_{k}(\phi_{t})$ is convex, because $\pL_{k}(\phi_{t})$ is convex (Proposition \ref{prop:7}) and $\pE(\phi_{t})$ is affine (Proposition \ref{prop:5}). Therefore, by convexity and by the variational formulas for $\pL_{k}$ (Proposition \ref{prop:2}) and $\pE$ (Proposition \ref{prop:1})
\begin{equation}\label{eq:12}
	\begin{split}
		\pF_{k}(\psi_{j})=\pF_{k}(\phi_{1})-\pF_{k}(\phi_{0})\geq&\frac{d}{dt}\Big |_{t=0^{+}}\pF_{k}(\phi_{t})\\
		&=\int_{X}(\frac{d}{dt}\Big |_{t=0^{+}}\phi_{t})(\beta_{k\phi_{0}}-\MA(\phi_{0})).
	\end{split}
\end{equation}
Now by Proposition \ref{prop:3} we have $\beta_{k\phi_{0}}-\MA(\phi_{0})\leq \delta_{k}\MA(\phi_{0})$ and by convexity of $\phi_{t}$ (Proposition \ref{prop:4} \emph{ii}) we find
\begin{displaymath}
	\frac{d}{dt}\Big |_{t=0^{+}}\phi_{t}\leq\phi_{1}-\phi_{0}\leq 0.
\end{displaymath}
Combined with \eqref{eq:12} we derive
\begin{displaymath}
	\pF_{k}(\psi_{j})\geq\delta_{k}\int_{X}(\frac{d}{dt}\Big |_{t=0^{+}}\phi_{t})\MA(\phi_{0}).
\end{displaymath}
Finally, we apply Proposition \ref{prop:5} \emph{ii} to get
\begin{displaymath}
	\pF_{k}(\psi_{j})\geq\delta_{k}(\pE(\phi_{1})-\pE(\phi_{0}))=\delta_{k}\pE(\psi_{j}).
\end{displaymath}
Therefore, if we let $j$ tend to $+\infty$, we obtain the desired inequality \eqref{eq:11} and hence \eqref{eq:10}. The proof is complete.
\end{proof}
\begin{remark}
\emph{i}. Along subgeodesics (independent of the imaginary part of the parameter) the operator $\pE$ is convex. However, the right property required in the proof is that $-\pE$ be convex. In general we can only ensure this along geodesics, which shows the necessity of this notion.

\emph{ii}. In view of the previous remark, it is tempting to try the argument of the proof with affine paths $\phi_{t}=(1-t)\phi_{0}+t\phi_{1}$ instead of geodesics. Indeed, along such paths $-\pE$ is actually convex. However, convexity of $\pL_{k}$ along $\phi_{t}$ is in general not guaranteed by Berndtsson's theorem.
 \end{remark}

\section{Metrics of finite energy and arithmetic intersection theory}\label{sec:heights}
\subsection{} Let $\pi:\pX\rightarrow\Spec\Int$ be an arithmetic variety, namely an integral flat projective scheme over $\Int$, with smooth generic fiber $\pX_{\QQ}$. The set of complex points $\pX(\CC)$ has a natural structure of smooth complex analytic space. We denote by $F_{\infty}:\pX(\CC)\rightarrow\pX(\CC)$ the antiholomorphic involution given by the action of complex conjugation. 

Let $\aL$ be an invertible sheaf on $\pX$. We will assume that $\aL_{\QQ}$ is semi-ample\footnote{Recall this means that a sufficiently big power $\aL_{\QQ}^{\otimes k}$ is generated by global sections.} and big over $\pX_{\QQ}$. Then, the elements of pluripotential theory developed in the preceding sections can be applied to $\pX(\CC)$ and $\aL_{\CC}$.
\begin{definition}
A semi-positive arakelovian metric of finite energy on $\aL$, or simply a metric of finite energy, is a singular hermitian metric on $\aL_{\CC}$ of the form $h_{\phi}=h_{0}e^{-\phi}$, where:
\begin{itemize}
	\item $h_{0}$ is a smooth hermitian metric on $\aL_{\CC}$ with semi-positive first Chern form $\omega_{0}:=\c1(\ov{\aL}_{0})$;
	\item $\phi\in\pE^{1}(\pX(\CC),\omega_{0})$;
	\item $h_{\phi}$ is invariant under the action of complex conjugation $F_{\infty}$.
\end{itemize}
\end{definition}
Because $\aL_{\QQ}$ is semi-ample, a smooth metric $h_{0}$ with the listed properties always exists, and may be choosen to be invariant under complex conjugation. We fix $h_{0}$ and $h=h_{\phi}$ once and for all, and write $\ov{\aL}_{0}$ and $\ov{\aL}$ for the corresponding hermitian line bundles. 

\subsection{} We assume given an invertible sheaf $\aK$ on $\pX$, coinciding over the generic fiber with the canonical sheaf: $\aK_{\QQ}=K_{\pX_{\QQ}}$. For every integer $k\geq 0$, the module of global sections $M_{k}:=H^{0}(\pX,\aL^{\otimes k}\otimes\aK)$ is a lattice in the finite dimensional real vector space $M_{k,\RR}:=H^{0}(\pX(\CC),\aL_{\CC}^{\otimes k}\otimes K_{\pX_{\CC}})^{F_{\infty}}$. According to the discussion \textsection \ref{subsection:det-coh}, attached to $h$ there is a natural $L^{2}$ euclidean structure $\|\cdot\|_{2}$ on $M_{k,\RR}$. We can thus compute the covolume of the lattice with respect to this structure. The \emph{arithmetic degree} of $(M_{k},\|\cdot\|_{2})$ is by definition
\begin{displaymath}
	\adeg H^{0}(\pX,\aL^{\otimes k}\otimes\aK)_{L^{2}}:=-\log \vol\left(\frac{M_{k,\RR}}{M_{k}}\right).
\end{displaymath}
We introduce an arithmetic counterpart of the functional $\pL_{k}$ defined in \eqref{eq:2}:
\begin{equation}\label{eq:15}
	\pL_{k}^{\ari}(\phi):=\frac{2}{kN_{k}}\adeg H^{0}(\pX,\aL^{\otimes k}\otimes\aK)_{L^{2}},
\end{equation}
where $N_{k}=\dim H^{0}(\pX_{\CC},\aL_{\CC}^{\otimes k}\otimes\aK_{\CC})$, $k\gg 0$, and we recall that $h=h_{\phi}=h_{0}e^{-\phi}$. Because $\aL_{\QQ}$ is nef and big, by the Kawamata-Viehweg vanishing theorem \cite[Thm. 4.3.1]{Lazarsfeld:I} we have $H^{i}(\pX_{\CC},\aL_{\CC}^{\otimes k}\otimes\aK_{\CC})=0$ for $i,k\geq 1$. The Riemann-Roch theorem then provides the estimate
\begin{displaymath}
	N_{k}=\frac{k^{d-1}}{(d-1)!}\deg\aL_{\CC}+o(k^{d-1}).\footnote{An elementary argument shows that, by nefness of $\aL_{\QQ}$ and projectivity of $\pX_{\QQ}$, the higher cohomology groups are $o(k^{d-1})$. This is enough to obtain the estimate.}
\end{displaymath}
We bring the reader's attention to the factor 2 in the definition \eqref{eq:15}, included to ensure the compatibility with the analytic $\pL_{k}$ functional dealt with so far. With this normalisation, given another $\omega_0$-psh function $\phi^{\prime}$ of finite energy, the following relation with the $\pL_{k}$ functional is readily shown:
\begin{equation}\label{eq:16}
	\pL_{k}^{\ari}(\phi)-\pL_{k}^{\ari}(\phi^{\prime})=\pL_{k}(\phi)-\pL_{k}(\phi^{\prime}).
\end{equation}
In this expression, the $\pL_{k}$ functional is for the line bundle $\aL_{\CC}$ on $\pX(\CC)$ and depends on the choice of a fixed orthonormal basis of $M_{k,\CC}$, with respect to the $L^{2}$ metric attached to $h_{0}$. In particular, by choosing $\phi^{\prime}=\phi_{0}=0$, we thus have
\begin{equation}
	\pL_{k}^{\ari}(\phi)=\pL_{k}^{\ari}(\phi_{0})+\pL_{k}(\phi).
\end{equation}
\subsection{}\label{subsec:heights} The height $h_{\ov{\aL}_{0}}(\pX)$ of $\pX$ with respect to $\ov{\aL}_{0}$ (more generally of any cycle in $\pX$) has been defined by Bost-Gillet-Soul\'e \cite[Sec. 3]{BGS} by means of higher dimensional arithmetic intersection theory. It is an arithmetic analogue of the degree of a variety. When the metric $h$ is logarithmically singular in the sense of Burgos-Kramer-K\"uhn \cite[Sec. 7]{BKK1}, the generalized arithmetic intersection theory of \emph{loc. cit.} still allows to define the height of $\pX$ with respect to $\ov{\aL}$. It satisfies
\begin{equation}\label{eq:13}
	h_{\ov{\aL}}(\pX)=h_{\ov{\aL}_{0}}(\pX)+\frac{1}{2}d(\deg\aL_{\CC})\pE(\phi),\quad d=\dim\pX,
\end{equation}
where the energy $\pE(\phi)$ is computed with respect to $\omega_{0}$ \cite[Prop. 6.5]{GFM}. We refer to \emph{loc. cit.} for a detailed study of these heights. More generally, if $h$ is an arbitrary semi-positive metric of finite energy, we define $h_{\ov{\aL}}(\pX)$ by equation \eqref{eq:13}. By the properties of the energy functional \cite[Cor. 4.2 and Rmk. 4.6]{Berman-Boucksom}, one proves with ease that $h_{\ov{\aL}}(\pX)$ is intrinsically defined, namely it only depends on $h$ and not on the smooth reference metric $h_0$. Observe however that, while $h_{\ov{\aL}}(\pX)$ can be defined, the height of a cycle in $\pX$ is in general meaningless. 

\begin{remark}\label{remark:height-extension}
Actually, formula \eqref{eq:13} and the monotonicity properties of the energy functional $\pE$ show that the height can be extended to singular semi-positive hermitian line bundles $\ov{\aL}$, just by declaring
\begin{displaymath}
	\begin{split}
		h_{\ov{\aL}_{\phi}}(\pX):&=\inf_{\psi\geq\phi}h_{\ov{\aL}_{\psi}}(\pX)\\
		&=h_{\ov{\aL}_{0}}(\pX)+\frac{1}{2}d(\deg\aL_{\CC})\inf_{\psi\geq\phi}\pE(\psi)\in\RR\cup\lbrace-\infty\rbrace.
	\end{split}
\end{displaymath}
Here the inf runs over the bounded $\omega_{0}$-psh functions, invariant under the action of complex conjugation. With this definition, the assignment $\phi\mapsto h_{\ov{\aL}_{\phi}}(\pX)$ is non-decreasing and continuous along point-wise decreasing sequences of $\omega_{0}$-psh functions, invariant under complex conjugation. Furthermore, this extension is uniquely determined by these properties. Therefore, the class of semi-positive hermitian metrics of finite energy is the biggest one for which the height can be defined and is a real number. 
\end{remark}

\subsection{Arithmetic Hilbert-Samuel theorems} Our first aim is to prove the following arithmetic analogue of the Hilbert-Samuel theorem in adjoint form.
\begin{theorem}\label{thm:HS}
Let $\pX\rightarrow\Spec\Int$ be an arithmetic variety of Krull dimension $d$ and $\ov{\aL}=(\aL,h)$ a semi-positive hermitian line bundle of finite energy. Assume $\aL_{\QQ}$ is semi-ample and big, and that $\aL$ is nef on vertical fibers. Let also $\aK$ be an invertible sheaf such that $\aK_{\QQ}=K_{\pX_{\QQ}}$. Then the following asymptotic expansion holds:
\begin{equation}\label{eq:14}
	\adeg H^{0}(\pX,\aL^{\otimes k}\otimes\aK)_{L^{2}}=h_{\ov{\aL}}(\pX)\frac{k^{d}}{d!}+o(k^{d})\text{ as }k\to+\infty.
\end{equation}
\end{theorem}
The proof of this statement reduces to the case of the smooth semi-positive hermitian metric $h_{0}$.
\begin{lemma}\label{lemma:reduction-smooth}
The asymptotic expansion \eqref{eq:14} holds for any semi-positive metric of finite energy as soon as it holds for some smooth semi-positive hermitian metric, for instance the reference metric $h_{0}$.
\end{lemma}
\begin{proof}
The expressions \eqref{eq:16} and \eqref{eq:13} show that the difference between the finite energy and the smooth cases is expressed solely in terms of $\pL_{k}(\phi)$ and $\pE(\phi)$. These are already controlled by Theorem \ref{thm:asymptotics-L-functional}. 
\end{proof}

By the lemma we can reduce to the case of the smooth semi-positive metric $h_0$. The proof is an adaptation of the argument of \cite[Thm. 1.4]{Zhang}. We will give the main lines of the proof. We start by recalling a lemma that will also be useful for Theorem \ref{thm:HS_cor} below.

\begin{lemma}\label{lemma:minima}
Let $M'\subset M'_{\RR}$ and $M\subset M_{\RR}$ be lattices in finite dimensional real vector spaces of dimensions $d'$, $d$, respectively. Assume $M'_{\RR}\subseteq M_{\RR}$ and $M'\subseteq M$. Let $\|\cdot\|$ be a euclidean metric on $M_{\RR}$, and use the same notation for its restriction to $M'_{\RR}$. Then
\begin{displaymath}
	\adeg (M,\|\cdot\|)-\adeg(M',\|\cdot\|)\geq-\log(d!)-(d-d')\log(\frac{1}{2}\lambda_{d}(M)).
\end{displaymath}
Here $\lambda_{d}(M)$ is defined as
\begin{displaymath}
	\lambda_{d}(M)=\inf\lbrace \sup\lbrace\|e_{1}\|,\ldots,\|e_{d}\|\rbrace\mid e_{1},\ldots e_{d}\in M \text{ independent}\rbrace.
\end{displaymath}
\end{lemma}
\begin{proof}
This is exactly \cite[Lemma (1.7)]{Zhang}.
\end{proof}
\begin{proof}[Proof of Theorem \ref{thm:HS}]
By Lemma \ref{lemma:reduction-smooth} we are led to treat the case of a smooth semi-positive metric $h_0$.

We fix a very ample line bundle $\aM$ on $\pX$, and we equip it with a smooth hermitian metric with strictly positive curvature form. The semi-ampleness of $\aL_{\QQ}$ and $\aM_{\QQ}$ guarantees that the graded algebra
\begin{displaymath}
 	S:=\bigoplus_{n,m\geq 0}H^{0}(\pX_{\QQ},\aL_{\QQ}^{\otimes n}\otimes\aM_{\QQ}^{\otimes m})
\end{displaymath}
is of finite type over $\QQ$ and
\begin{displaymath}
 \bigoplus_{n,m\geq 0}H^{0}(\pX_{\QQ},\aL_{\QQ}^{\otimes n}\otimes\aM_{\QQ}^{\otimes m}\otimes\aK_{\QQ})
\end{displaymath}
is a finite graded $S$-module. This implies, as in \cite[Lemma (1.6)]{Zhang}, that there exists a real constant $c>1$ such that every $\Int$-module  $H^{0}(\pX,\aL^{\otimes n}\otimes\aM^{\otimes m}\otimes\aK)$ contains a set of sections of maximal rank with $L^{\infty}$ and hence $L^2$ norm bounded by $c^{1+\max(n,m)}$.

We fix a non-vanishing global section $s$ of $\aM$ with sup norm bounded by a positive constant $c'$. The line bundle $\aL_{\QQ}$ is big by assumption. Then, Kodaira's lemma \cite[Prop. 2.2.6]{Lazarsfeld:I} states there exists $n\geq 1$ such that
\begin{displaymath}
 	H^{0}(\pX_{\QQ},\aL_{\QQ}^{\otimes n}\otimes\aM_{\QQ}^{-1})\neq 0.
\end{displaymath}
Let $t$ be such a non-vanishing global section, whose sup norm is bounded by a constant $c''$. 

Let $k>n$ and $j$ be any positive integers, and $i$ an integer between 0 and $k-1$. We consider $H^{0}(\pX,\aL^{\otimes (kj+i)}\otimes\aK)$. Tensoring with $s^{j}$ defines a monomorphism
\begin{displaymath}
	\begin{split}
		\alpha:M_{1}=H^{0}(\pX,\aL^{\otimes (kj+i)}&\otimes\aK)\\
		&\hookrightarrow M_{2}=H^{0}(\pX,\aL^{\otimes (kj+i)}\otimes\aM^{\otimes j}\otimes\aK),
	\end{split}
\end{displaymath}
whose norm is bounded by $c^{\prime j}$. Similarly, multiplication by $t^{j}$ gives a monomorphism
\begin{displaymath}
	\begin{split}
	\beta:M_{3}=H^{0}(\pX,\aL^{\otimes ((k-n)j+i)}\otimes\aM^{\otimes j}&\otimes\aK)\\
	&\hookrightarrow M_{1}=H^{0}(\pX,\aL^{\otimes (kj+i)}\otimes\aK),
	\end{split}
\end{displaymath}
whose norm is bounded by $c^{\prime\prime j}$. Applying Lemma \ref{lemma:minima}, we obtain inequalities
\begin{equation}\label{eq:60}
	\begin{split}
		\adeg \ov{M}_{2,L^{2}}-\adeg\ov{M}_{1,L^{2}}\geq & -\log((\rk M_{2})!)\\
		&\hspace{0.3cm}-(\rk M_{2}-\rk M_{1})(1+k(j+1))\log(c/2)\\
		&\hspace{0.6cm}-j(\rk M_{1})\log(c')
	\end{split}
\end{equation}
 and
 \begin{equation}\label{eq:61}
 	\begin{split}
		\adeg \ov{M}_{1,L^{2}}-\adeg\ov{M}_{3,L^{2}}\geq & -\log((\rk M_{1})!)\\
		&\hspace{0.3cm}-(\rk M_{1}-\rk M_{3})(1+k(j+1))\log(c/2)\\
		&\hspace{0.6cm}-j(\rk M_{3})\log(c'').
	\end{split}
 \end{equation}
 By the Riemann-Roch theorem and the arithmetic Riemann-Roch theorem for smooth positive hermitian line bundles, we derive the asymptotics
 \begin{equation}\label{eq:62}
 	\rk M_{l}=\frac{(jk+i)^{d-1}}{(d-1)!}+O(k^{d-1}j^{d-1})+o_{k}(j^{d-1}),\quad l=1,2,3,
\end{equation}
\begin{equation}\label{eq:63}
	\adeg\ov{M}_{l,L^{2}}=\frac{(jk+i)^{d}}{d!}h_{\ov{\aL}}(\pX)+O(k^{d-1}j^{d})+o_{k}(j^{d}),\quad l=2,3.
 \end{equation}
 The notation $o_{k}$ represents a little $o$ quantity depending on $k$. 
 
 To conclude, let $\varepsilon>0$ and fix $k>n$ such that $k^{d-1}j^{d}<\varepsilon k^{d}j^{d}$ for all $j\geq 1$. For $j$ sufficiently big, and since $k$ has being fixed, $o_{k}(j^{d})$ will be bounded by $\varepsilon k^{d}j^{d}$. Then, from \eqref{eq:60}--\eqref{eq:63} we see that
 \begin{displaymath}
 	\left|\adeg\ov{M}_{1,L^{2}}-\frac{(kj+i)^{d}}{d!}h_{\ov{\aL}}(\pX)\right| \leq\kappa_{0}\varepsilon (kj+i)^{d},
 \end{displaymath}
 where $\kappa_{0}$ is a positive constant independent of $i,j$. The proof is complete.
\end{proof}
For log-singular hermitian line bundles, which are in particular of finite energy, we have a more general result.
\begin{theorem}\label{thm:HS_cor}
Let $\pX\rightarrow\Spec\Int$ be an arithmetic variety of Krull dimension $d$ and $D\subset\pX_{\QQ}$ a divisor with normal crossings. Let $\ov{\aL}$ be a semi-positive log-singular hermitian line bundle, with singularities along $D(\CC)$. Assume $\aL_{\QQ}$ is semi-ample and big, and that $\aL$ is nef on vertical fibers. Let also $\ov{\aN}$ be an arbitrary log-singular hermitian line bundle with singularities along $D(\CC)$. Then there is an asymptotic expansion
\begin{displaymath}
 	\adeg H^{0}(\pX,\aL^{\otimes k}\otimes\aN)_{L^2}=h_{\ov{\aL}}(\pX)\frac{k^{d}}{d!}+o(k^{d})\text{ as }k\to+\infty,
\end{displaymath}
where the $L^{2}$ norms are computed with respect to any smooth volume form $\mu$ on $\pX(\CC)$, invariant under the action of complex conjugation.
\end{theorem}
The reader is referred to \cite[Sec. 7]{BKK1} and \cite[Def. 3.29]{BKK2} for the definition and first properties of log-singular line bundles. For a more detailed study of this and related notions in arithmetic intersection theory, the reader can consult \cite{GFM}.

The first step in the direction of the proof of the theorem is the following lemma, which is an analogue of \cite[Lemma (1.6)]{Zhang} for log-singular hermitian line bundles.

\begin{lemma}\label{lemma:bounded_norms}
Let $\pX$ be an arithmetic variety of Krull dimension $d$ and $D\subset\pX(\CC)$ a divisor with normal crossings. Let $\ov{\aL}$ be a log-singular hermitian line bundle over $\pX$, with singularities along $D$, which is semi-ample on the generic fiber. Let $\ov{\aM}$ be any log-singular hermitian line bundle on $\pX$, with singularities along $D$ as well. Fix a smooth volume form $\mu$, invariant under the action of complex conjugation, with respect to which we compute $L^{2}$ norms. Then there exists a real positive constant $C>0$ and an integer $N\geq 1$ such that for every $k\geq 0$, the $\Int$-module $H^{0}(\pX,\aL^{\otimes k}\otimes\aM)$ contains a set of independent sections of maximal rank whose $L^2$ norm is bounded by $C^{k+1}((k+1)N)!^{(d -1)/2}$.
\end{lemma}
\begin{proof}
Because $\aL_{\QQ}$ is semi-ample, the graded $\QQ$-algebra
\begin{displaymath}
	S=\bigoplus_{k\geq 0} H^{0}(\pX_{\QQ},\aL^{\otimes k}_{\QQ})
\end{displaymath}
if of finite type. Also, the $\QQ$-vector space
\begin{displaymath}
	N=\bigoplus_{k\geq 0}H^{0}(\pX_{\QQ},\aL^{\otimes k}_{\QQ}\otimes\aM_{\QQ})
\end{displaymath}
is an $S$-module of finite type. Let us write
\begin{align*}
	&S=\QQ[s_{1},\ldots,s_{r}],\\
	&N=Sn_{1}+\ldots+Sn_{l},
\end{align*}
where the sections $s_{i}$ are homogenous of degrees $d_{i}\geq 1$ and the sections $n_{j}$ are homogenous of degrees $e_{j}$. We can suppose that the sections $s_{i}$ and $n_{j}$ are actually integral, by clearing denominators. Given integers $a_{1},\ldots,a_{r}\geq 0$, we are thus reduced to bound the $L^{2}$ norms of sections of the form
\begin{displaymath}
	s_{1}^{a_{1}}\ldots s_{r}^{a_{r}}n_{j}.
\end{displaymath}
For this, it is convenient to write the log-singular metrics in terms of smooth metrics. Namely, we write
\begin{align}
	&\|\cdot\|_{\ov{\aL}}^{2}=\|\cdot\|_{0}^{2}\varphi_{0},\\
	&\|\cdot\|_{\ov{\aL}^{\otimes e_{j}}\otimes \aM}^{2}=\|\cdot\|_{j}^{2}\varphi_{j},\quad j=1,\ldots,l.
\end{align}
Here the metrics $\|\cdot\|_{0}$ and $\|\cdot\|_{j}$ are smooth. The functions $\varphi$, $\varphi_{j}$, $j=1,\ldots,l$ are smooth on $\pX(\CC)\setminus D$ and have logarithmic growth along $D$. We can cover $D$ by a finite number of coordinate polydisks of radius $2\varepsilon<2/e$, such that the concentric polydisks of radius $\varepsilon$ still form a cover. Furthermore, we can suppose that in the corresponding coordinates $z_{1},\ldots,z_{d-1}$ ($d-1=\dim\pX(\CC)$), we have
\begin{displaymath}
	\varphi,\varphi_{j}\leq B\left(\prod_{i=1}^{d-1}\log |z_{i}|^{-1}\right)^{N},\quad j=1,\ldots,l
\end{displaymath}
on the range $|z_{i}|\leq\varepsilon$, $i=1,\ldots, d-1$, and for some constant $B>0$ and integer $N\geq 1$. 

For the $L^{2}$ norms we have the inequality
\begin{equation}\label{eq:70}
	\begin{split}
	\|s_{1}^{a_{1}}\ldots s_{r}^{a_{r}}n_{j}\|_{L^{2}}^{2}\leq \|s_{1}\|_{0,\infty}^{2a_{1}}&\ldots \|s_{r}\|_{0,\infty}^{2a_{r}}\|n_{j}\|_{j,\infty}^{2}\\
	&\cdot\int_{\pX(\CC)}\varphi^{d_{1}a_{1}+\ldots+d_{r}a_{r}}\varphi_{j}d\mu.
	\end{split}
\end{equation}
We proceed to bound the integral in \eqref{eq:70}. Let us introduce $D=1+\sum_{i}d_{i}a_{i}$. On a coordinate polydisk as before we have
\begin{equation}\label{eq:71}
	\begin{split}
	\int_{D(0,\varepsilon)^{d-1}}\varphi^{d_{1}a_{1}+\ldots+d_{r}a_{r}}&\varphi_{j}d\mu\leq \\
	C' B^{D}
	&\prod_{i=1}^{d-1}\int_{D(0,\varepsilon)}(\log |z_{i}|^{-1})^{ND}|dz_{i}\wedge d\cz_{i}|
	\end{split}
\end{equation}
for some constant $C'>0$ determined by the local expression of the volume form $\mu$. A computation in polar coordinates shows that the right hand side of \eqref{eq:71} can further be bounded by
\begin{displaymath}
	(4\pi)^{d-1}C' \varepsilon^{2}(B\log\varepsilon^{-1})^{DN}(DN)!^{d-1}.
\end{displaymath}
After adjusting $C'$ (depending only on the $\|s_{i}\|_{0,\infty}$ $\|n_{j}\|_{j,\infty}$, $B$, $N$, $\log\varepsilon^{-1}$, the number of polydisks covering $D$ and the integrals of the $\varphi^{D-1}\varphi_{j}$ on the complement of the polydisks, where $\varphi$ and $\varphi_{j}$ are bounded), we finally find
\begin{displaymath}
		\|s_{1}^{a_{1}}\ldots s_{r}^{a_{r}}n_{j}\|_{L^{2}}^{2}\leq C'^{D}(DN)!^{d-1}.
\end{displaymath}
We conclude after observing that $D\leq k+1$ and putting $C=(C'+1)^{1/2}$.
\end{proof}
\begin{remark}
The lemma does not require the divisor $D$ to be defined over $\QQ$, nor to have \emph{strict} normal crossings. Recall that strict normal crossing divisors have smooth irreducible components, while the components of normal crossing divisors can have self-intersections.
\end{remark}
\begin{proof}[Proof of Theorem \ref{thm:HS_cor}]
First of all, because two smooth volume forms are comparable, we can suppose that the fixed volume form $\mu$ comes from a smooth hermitian metric on $T_{\pX(\CC)}$. In its turn, this metric induces a smooth hermitian metric on $\aK$. Notice that, with these choices, the $L^{2}$ structure on $H^{0}(\pX,\aL^{\otimes k}\otimes\aK)$ defined according to \textsection \ref{subsection:det-coh} and the $L^{2}$ structure defined using the metric on $\ov{\aL}^{\otimes k}\otimes\ov{\aK}$ and the volume form $\mu$ coincide.

Since $\aL_{\QQ}$ is semi-ample, we can apply the lemma with $\ov{\aM}=\ov{\aK}$, $\ov{\aN}$. We can choose common constants $C$, $N$ for both choices. Furthermore, bigness of $\aL_{\QQ}$ ensures there exist $k_{0}, k_{1}\geq 1$ such that
\begin{align*}
	&H^{0}(\pX_{\QQ},\aL^{\otimes k_{0}}_{\QQ}\otimes\aK_{\QQ}\otimes\aN^{-1}_{\QQ}(-D))\neq 0\\
	&H^{0}(\pX_{\QQ},\aL^{\otimes k_{1}}_{\QQ}\otimes\aK^{-1}_{\QQ}\otimes\aN_{\QQ}(-D))\neq 0.
\end{align*}
(we considered $D$ with its reduced scheme structure). We choose respective non-vanishing global sections $t_0$ and $t_{1}$. After clearing denominators, we can suppose that
\begin{align*}
	&t_{0}\in H^{0}(\pX,\aL^{\otimes k_{0}}\otimes\aK\otimes\aN^{-1})\\
	&t_{1}\in H^{0}(\pX,\aL^{\otimes k_{1}}\otimes\aK^{-1}\otimes\aN).
\end{align*}
A crucial observation for the sequel is that the pointwise norms $\|t_{0}\|$ and $\|t_{1}\|$ are actually continuous on $\pX(\CC)$, and even vanishing along $D(\CC)$. Indeed, this follows because $t_{0}$ and $t_{1}$ vanish along $D$ and the metrics are log-singular: the holomorphic vanishing of $t_{0}$, $t_{1}$ along $D(\CC)$ cancels the logarithmic singularities of the metrics. Hence $\|t_{0}\|_{\infty},\|t_{1}\|_{\infty}<\infty$.

Multiplication by $t_{0}$ defines an injective morphism between hermitian modules
\begin{displaymath}
	\xymatrix{
		\iota_{0}:H^{0}(\pX,\aL^{\otimes k}\otimes\aN)_{L^{2}}\ar@{^{(}->}[r]	&H^{0}(\pX,\aL^{\otimes(k+k_{0})}\otimes\aK)_{L^{2}}.
	}
\end{displaymath}
whose norm is bounded by $\|t_{0}\|_{\infty}$. The corresponding higher cohomology groups have rank $O(k^{d-2})$, by semi-ampleness (hence nefness) of $\aL_{\QQ}$. Therefore, by the Riemann-Roch theorem, the cokernel of $\iota_{0}$ has rank $O(k^{d-2})$ (the leading terms of the Hilbert-Samuel polynomials are equal). Similarly, there is an injective morphism
\begin{displaymath}
	\xymatrix{
		\iota_{1}:H^{0}(\pX,\aL^{\otimes k}\otimes\aK)_{L^{2}}\ar@{^{(}->}[r]	&H^{0}(\pX,\aL^{\otimes(k+k_{1})}\otimes\aN)_{L^{2}}
	}
\end{displaymath}
whose norm is bounded by $\|t_{1}\|_{\infty}$, and with cokernel of rank $O(k^{d-2})$. We combine these facts together with Lemma \ref{lemma:minima} and the estimate provided by Lemma \ref{lemma:bounded_norms} , to show that there exists a constant $A>0$ such that
\begin{align*}
	&\adeg H^{0}(\pX,\aL^{k+k_{0}}\otimes\aK)_{L^{2}}-\adeg H^{0}(\pX,\aL^{\otimes k}\otimes\aN)_{L^{2}}
		\geq -A k^{d-1}\log k,\\
	&\adeg H^{0}(\pX,\aL^{k+k_{1}}\otimes\aN)_{L^{2}}-\adeg H^{0}(\pX,\aL^{\otimes k}\otimes\aK)_{L^{2}}
		\geq -A k^{d-1}\log k.
\end{align*}
These inequalities together with Theorem \ref{thm:HS} are enough to conclude.
\end{proof}
\begin{remark}
\emph{i}. The rationality of the divisor $D$ allows to chose integral sections $t_0$ and $t_1$ that vanish along $D$. The argument breaks down if $D$ is only supposed to be defined over $\CC$.

\emph{ii}. It follows from the very proof of the theorem that the hypothesis on the metrics can be weakened. It is enough to deal with metrics of finite energy, that are locally bounded on the complement of $D(\CC)$ and with logarithmic growth near $D(\CC)$. The arguments are exactly the same.
\end{remark}

\section{Arithmetic volumes of integral cusp forms on Hilbert modular surfaces}\label{sec:Hilbert}
In this section we apply Theorem \ref{thm:HS} to suitable arithmetic models of toroidal compactifications of Hilbert modular surfaces. Combining with the main results of Bruinier-Burgos-K\"uhn \cite{BBK}, we are able to prove an arithmetic analogue of the classical theorem expressing the dimension of the spaces of cusps forms in terms of a special value of a Dedekind zeta function. To shorten the presentation, we refer to the book \cite{vdG} on Hilbert modular surfaces. We also make use of the results of Rapoport \cite{Rapoport} and Chai \cite{Chai} on arithmetic toroidal compactifications. The arithmetic theory of Hilbert modular surfaces is also summarized in Bruinier-Burgos-K\"uhn \cite[Sec. 5]{BBK}.

We fix a real quadratic number field $F$ of discriminant $D\equiv 1\mod 4$.\footnote{This hypothesis is only required for computations of arithmetic intersection numbers.} Let $\ell$ be a positive integer. Write $\Gamma_{F}(\ell)\subset\SL_{2}(\OO_{F})$ for the principal level $\ell$ congruence subgroup. Let $\mathcal{H}(\ell)$ be the smooth algebraic stack over $\Spec\Int[\zeta_{\ell},1/\ell]$ classifying $\mathfrak{d}^{-1}$ polarized abelian surfaces with multiplication by $\OO_{F}$ and principal level $\ell$ structure\footnote{As usual, we denote by $\mathfrak{d}^{-1}$ the inverse different of $F$, namely the inverse fractional ideal of $\mathfrak{d}=(\sqrt{D})\subset\OO_{F}.$}. The algebraic stack $\mathcal{H}(\ell)$ carries a universal family of abelian schemes. We consider an arithmetic toroidal compactification $\overline{\mathcal{H}}(\ell)$ of $\mathcal{H}(\ell)$. This is an irreducible, smooth and proper algebraic stack over $\Spec\Int[\zeta_{\ell},1/\ell]$, with geometrically connected fibers. It comes equipped with a universal semi-abelian scheme extending the universal family over  $\mathcal{H}(\ell)$. We denote by $\omega_{\ell}$ the dual of the determinant of the relative Lie algebra of the universal semi-abelian scheme. It is possible to chose toroidal data giving rise to a whole tower $\lbrace\ov{\mathcal{H}}(\ell)\rbrace_{\ell\geq 1}$, whose constituents are projective schemes whenever $\ell\geq 3$. The arrows of the tower are of the form
\begin{displaymath}
	\pi_{\ell,\ell'}:\ov{\mathcal{H}}(\ell)\rightarrow\ov{\mathcal{H}}(\ell')[1/\ell],
\end{displaymath}
whenever $\ell'\mid \ell$, and extend the natural projections ``forgetting level structure" $\mathcal{H}(\ell)\rightarrow\mathcal{H}(\ell')[1/\ell]$. The notation $[1/\ell]$ indicates that we inverted the primes dividing $\ell$ in the structure sheaf of a scheme. Furthermore, we have the compatibility
\begin{equation}\label{eq:1000}
	\pi_{\ell,\ell'}^{\ast}\omega_{\ell'}=\omega_{\ell}.
\end{equation}
For $\ell\geq 3$, the sheaves $\omega_{\ell}$ are semi-ample relative to $\Spec\Int[\zeta_{\ell}, 1/\ell]$ and big over the generic fiber. Observe that bigness is a consequence of the existence of a natural birational morphism to the minimal compactification:
\begin{displaymath}
	\pi:\ov{\mathcal{H}}(\ell)_{\QQ}\longrightarrow\mathcal{H}(\ell)^{\ast}_{\QQ},
\end{displaymath}
such that $\pi^{\ast}\OO(1)=\omega_{\ell\,\QQ}^{\otimes m}$ for some very ample line bundle $\OO(1)$. In the sequel, we will consider $\ov{\mathcal{H}}(\ell)$ just as a scheme over $\Spec\Int[1/\ell]$.

To apply our results, we need to extend schematic toroidal compactifications to $\Spec\Int$. For this, we fix from now on relatively prime integers $N,M\geq 3$ and put $\ell=NM$. We chose toroidal data so that $\ov{\mathcal{H}}(\ell)$, $\ov{\mathcal{H}}(N)$ and $\ov{\mathcal{H}}(M)$ are schemes and arise in a tower as above (actually only the levels $N,M,\ell$ are involved).
\begin{lemma}
There exists a projective integral scheme $\mathcal{X}\rightarrow\Spec\Int$, together with a semi-ample invertible sheaf $\aL$, fulfilling the following properties:

i. $\mathcal{X}[1/\ell]=\ov{\mathcal{H}}(\ell)$ and $\aL\mid_{\mathcal{X}[1/\ell]}=\omega_{\ell}$;

ii. there are proper morphisms
\begin{displaymath}
	p_{N}\colon\mathcal{X}[1/N]\rightarrow\ov{\mathcal{H}}(N),\quad p_{M}\colon\mathcal{X}[1/M]\rightarrow\ov{\mathcal{H}}(M),
\end{displaymath}
such that $p_{N}\mid_{\mathcal{X}[1/\ell]}=\pi_{\ell,N}$ and $p_{M}\mid_{\mathcal{X}[1/\ell]}=\pi_{\ell,M}$. Furthermore,
\begin{displaymath}
	 \aL\mid_{\mathcal{X}[1/N]}=p_{N}^{\ast}\omega_{N},\quad \aL\mid_{\mathcal{X}[1/M]}=p_{M}^{\ast}\omega_{M}.
\end{displaymath}
\end{lemma}
\begin{proof}
Let $r\in\lbrace N,M,\ell\rbrace$. Fix closed embeddings $\ov{\mathcal{H}}(r)\hookrightarrow\PP_{\Int[1/r]}^{d_r}$. We define $\widetilde{\mathcal{H}}(r)$ to be the Zariski closure of $\ov{\mathcal{H}}(r)$ in $\PP_{\Int}^{d_r}$. The projections $\pi_{\ell,N}$ and $\pi_{\ell,M}$ induce a natural morphism
\begin{displaymath}
	\varphi:\ov{\mathcal{H}}(\ell)\longrightarrow\widetilde{\mathcal{H}}(N)[1/\ell]\times\widetilde{\mathcal{H}}(M)[1/\ell],
\end{displaymath}
since $\widetilde{\mathcal{H}}(N)[1/\ell]=\ov{\mathcal{H}}(N)[1/\ell]$ and $\widetilde{\mathcal{H}}(M)[1/\ell]=\ov{\mathcal{H}}(M)[1/\ell]$. Let $\Gamma$ denote the graph of $\varphi$. We define $\mathcal{X}$ as the Zariski closure of $\Gamma$ inside $\widetilde{\mathcal{H}}(\ell)\times\widetilde{\mathcal{H}}(N)\times\widetilde{\mathcal{H}}(M)$. There are natural proper morphisms
\begin{displaymath}
	q_{r}:\mathcal{X}\longrightarrow\widetilde{\mathcal{H}}(r),\quad r=N,M,\ell.
\end{displaymath} 
Because $\widetilde{\mathcal{H}}(r)[1/r]=\ov{\mathcal{H}}(r)$, the $q_r$ induce proper projections
\begin{displaymath}
	p_{r}:\mathcal{X}[1/r]\longrightarrow\ov{\mathcal{H}}(r),\quad r=N,M,\ell.
\end{displaymath}
Consider the line bundles $\aL_{r}=p_{r}^{\ast}\omega_{r}$ over $\mathcal{X}[1/r]$, for all $r$. By construction and by the compatibility \eqref{eq:1000}, we have
\begin{displaymath}
	\aL_{r}\mid_{\mathcal{X}[1/\ell]}=\omega_{\ell}.
\end{displaymath}
Therefore these line bundles glue into a single line bundle $\aL$ over $\mathcal{X}$. Moreover, because $\omega_{r}$ is semi-ample for all $r$, it follows that $\aL$ is semi-ample too. It is readily checked that $\mathcal{X}$, $\aL$, $p_{N}$, $p_{M}$ satisfy the requirements of the statement.
\end{proof}
\begin{notation}
Let $\mathcal{X}$, $\aL$, $p_{N}$, $p_{M}$ be as in the lemma. We abusively write $\ov{\mathcal{H}}(\ell)=\mathcal{X}$, $\omega_{\ell}=\aL$ over $\Spec\Int$. We will call the data $(\ov{\mathcal{H}}(\ell),\omega_{\ell},p_{N},p_{M})$ a naive integral toroidal compactification over $\Spec\Int$.
\end{notation}
\begin{remark}
\emph{i}. The scheme $\ov{\mathcal{H}}(\ell)$ and the sheaf $\omega_{\ell}$ over $\Spec\Int$ are not canonically defined, not even in terms of the toroidal data: involved in their construction, there are arbitrary choices of closed embeddings of toroidal compactifications into projective spaces. However, any data $(\ov{\mathcal{H}}(\ell),\omega_{\ell},p_{N},p_{M})$ will be enough for our purposes.

\emph{ii}. Naive integral toroidal compactifications are not smooth schemes. To apply our arithmetic Hilbert-Samuel theorem, it is enough to know generic smoothness, which is the case.
\end{remark}
The invertible sheaf $\omega_{\CC}$ over $\ov{\mathcal{H}}(\ell)_{\CC}$ can be endowed with the so called (pointwise) Petersson metric \cite[Sec. 2.2]{BBK}. This is a semi-positive log-singular hermitian metric, with singularities along the normal crossing divisor $\ov{\mathcal{H}}(\ell)(\CC)\setminus\mathcal{H}(\ell)(\CC)$ \cite[Prop. 2.5]{BBK}. In particular, this is an example of semi-positive metric of finite energy. This metric is compatible with pull-back by the natural projections $\pi_{\ell,N}$. The corresponding $L^{2}$ metric on global sections $H^{0}(\ov{\mathcal{H}}(\ell)_{\CC},\omega^{\otimes k}_{\ell\,\CC}\otimes K_{\ov{\mathcal{H}}(\ell)_{\CC}})$ can be identified with the Petersson pairing on cusp forms of parallel weight $2k+2$. For this, recall the Kodaira-Spencer canonical isomorphism $\omega\simeq  K_{\ov{\mathcal{H}}(\ell)_{\CC}}(D)$, where $D$ is the divisor $D=\ov{\mathcal{H}}_{\CC}\setminus\mathcal{H}_{\CC}$. We write $\pet$ to refer to the Petersson pairing.

\begin{theorem}\label{thm:Hilbert}
Let $\ell=NM$ be the product of two coprime integers $N,M\geq 3$. Let $(\ov{\mathcal{H}}(\ell),\omega_{\ell},p_{N},p_{M})$ be a naive integral toroidal compactification over $\Spec\Int$. Endow $\omega_{\ell}$ with the Petersson metric. Let $\aK$ be a line bundle with $\aK_{\QQ}=K_{\ov{\mathcal{H}}(\ell)_{\QQ}}$. Then
\begin{displaymath}
	\begin{split}
	\adeg H^{0}(\ov{\mathcal{H}}(\ell),&\omega_{\ell}^{\otimes k}\otimes\aK)_{\pet}=\\
	&-\frac{k^{3}}{6}d_{\ell}\zeta_{F}(-1)\left(\frac{\zeta_{F}^{\prime}(-1)}{\zeta_{F}(-1)}+\frac{\zeta^{\prime}(-1)}{\zeta(-1)}+\frac{3}{2}+
	\frac{1}{2}\log D\right) +o(k^{3}),
	\end{split}
\end{displaymath}
where $d_{\ell}=[\QQ(\zeta_{\ell}):\QQ][\Gamma_{F}(1):\Gamma_{F}(\ell)]$, and $\zeta_{F}$ is the Dedekind zeta function of $F$.
\end{theorem}
\begin{proof}
Because the sheaf $\omega_{\ell}$ is semi-ample and $\omega_{\ell\, \QQ}$ is big and the Petersson metric has finite energy, we can apply Theorem \ref{thm:HS}. Hence, it is enough to check the formula
\begin{equation}\label{eq:40}
	h_{\ov{\omega}_{\ell}}(\ov{\mathcal{H}}(\ell))=-d_{\ell}\zeta_{F}(-1)\left(\frac{\zeta_{F}^{\prime}(-1)}{\zeta_{F}(-1)}+\frac{\zeta^{\prime}(-1)}{\zeta(-1)}+\frac{3}{2}+
	\frac{1}{2}\log D\right).
\end{equation}
Let $r\in\lbrace N,M\rbrace$. By the properties of naive toroidal compactifications and the compatibility of the Petersson metric with pull-back by $\pi_{\ell,N}$ and $\pi_{\ell,N}$, we see that
\begin{displaymath}
	p_{r}^{\ast}\ov{\omega}_{r}=\ov{\omega}_{\ell}\quad\text{on}\quad\ov{\mathcal{H}}(\ell)[1/r].
\end{displaymath}
By the functoriality properties of heights, we derive an equality in  $\RR/\sum_{p\mid r}\QQ\log p$
\begin{equation}\label{eq:41}
	\begin{split}
		h_{\ov{\omega}_{\ell}}(\ov{\mathcal{H}}(\ell))=&h_{p_{r}^{\ast}\ov{\omega}_{r}}(\ov{\mathcal{H}}(\ell))\\
		&=(\deg p_{r\,\QQ})h_{\ov{\omega}_{r}}(\ov{\mathcal{H}}(r))\\
		&\hspace{0.5cm}=-[\Gamma_{F}(r):\Gamma_{F}(\ell)]d_{r}\zeta_{F}(-1)\\
		&\hspace{1cm}\cdot\left(\frac{\zeta_{F}^{\prime}(-1)}{\zeta_{F}(-1)}+\frac{\zeta^{\prime}(-1)}{\zeta(-1)}+\frac{3}{2}+
	\frac{1}{2}\log D\right).
	\end{split}
\end{equation}
In the last equality we appealed to \cite[Thm. 6.4]{BBK}. Notice that $[\Gamma_{F}(r):\Gamma_{F}(\ell)]d_{r}=d_{\ell}$. Because $N,M$ are coprime, the relation \eqref{eq:41} for $r=N,M$ (with values in $\RR/\sum_{p\mid r}\log p$) implies the desired equality \eqref{eq:40} (with values in $\RR$). We conclude by Theorem \ref{thm:HS}.
\end{proof}
\bibliographystyle{amsplain}
\bibliography{biblioHS}
\end{document}